\documentclass[reqno,10pt]{amsart}
\usepackage{CJK}
\usepackage{hyperref}
\usepackage{cite}
\usepackage{amsmath}
\usepackage{mathrsfs}

\makeatletter
\@namedef{subjclassname@2020}{%
	\textup{2020} Mathematics Subject Classification}
\makeatother

\numberwithin{equation}{section}

\newtheorem{theorem}{Theorem}[section]
\newtheorem{lemma}[theorem]{Lemma}
\newtheorem{definition}[theorem]{Definition}

\newtheorem{remark}[theorem]{Remark}

\allowdisplaybreaks

\begin{document}
	
	\title[Weak solutions to a parabolic nonlocal 1-Laplacian equation] {Existence and uniqueness of weak solutions to a parabolic nonlocal 1-Laplacian equation}
	
	\author[D. Li and C. Zhang  \hfil \hfilneg]
	{Dingding Li and Chao Zhang$^*$}
	
	\thanks{$^*$ Corresponding author.}
	
	\address{Dingding Li \hfill\break School of Mathematics, Harbin Institute of Technology, Harbin 150001, China}
	\email{a87076322@163.com}

	\address{Chao Zhang  \hfill\break School of Mathematics and Institute for Advanced Study in Mathematics, Harbin Institute of Technology, Harbin 150001, China}
	\email{czhangmath@hit.edu.cn}

	
	\maketitle
	
\begin{abstract}
We consider a class of parabolic nonlocal  $1$-Laplacian equation
\begin{align*}
u_t+(-\Delta)^s_1u=f \quad \text{ in }\Omega\times(0,T].
\end{align*}
By employing the Rothe time-discretization method, we establish the existence and uniqueness of weak solutions to the equation above. In particular, different from the previous results on the local case, we infer that the weak solution maintains $\frac{1}{2}$-H\"{o}lder continuity in time.
\end{abstract}

\section{Introduction}
\thispagestyle{empty}
\label{sec1}

Suppose that $\Omega$ is a bounded domain of $\mathbb{R}^N(N\ge 2)$ with smooth boundary $\partial\Omega$ and $T$ is a positive number. Denote $\Omega_T:=\Omega\times (0, T]$. In this paper, we deal with the following fractional 1-Laplacian problem
\begin{align}
	\label{1.1}
	\begin{cases}
		u_t+\left( -\Delta\right)^s_1u=f &\text{in }\Omega_T,\\
		u(x,0)=u_0 &\text{on }\Omega,\\
		u=0 &\text{in }\mathcal{C}_\Omega\times(0,T],
	\end{cases}
\end{align}
where $\mathcal{C}_\Omega:=\mathbb{R}^N\backslash\Omega$. The operator $(-\Delta)^s_p$, called the fractional $p$-Laplacian, is given by
\begin{align*}
	&\left( -\Delta\right)^s_pu(x)=2\text{P.V.}\int_{\mathbb{R}^N}\frac{\left| u(x)-u(y)\right|^{p-2}\left( u(x)-u(y)\right)  }{|x-y|^{N+sp}}dy, \quad x\in \Omega
\end{align*}
with $p\in \left[ 1,+\infty\right) $ and $s\in(0,1)$. P.V. is known as an abbreviation in the sense of principal value.

Problem \eqref{1.1} could be seen as a nonlocal counterpart of the following 1-Laplacian Dirichlet problem
\begin{align}
	\label{1.2}
	\begin{cases} 
		\displaystyle
		u_t-\text{div}\left(\frac{\nabla u}{|\nabla u|}\right) =f &\text{in }\Omega_T,\\
		u(x,0)=u_0 &\text{on }\Omega,\\
		u=0 &\text{on }\partial\Omega\times(0,T].
	\end{cases}
\end{align}
As to the homogeneous case that $f\equiv0$ in $\Omega_T$, Andrew, Ballester, Caselles and Maz\'{o}n in \cite{ABCM01} introduced the concept of solutions to make clear the quotient $\frac{Du}{|Du|}$, even if $Du$ vanishes. Moreover, they used nonlinear semigroup theory to prove the existence and uniqueness of solutions and found a vector field $Z\in L^\infty(\Omega,\mathbb{R}^N)$, satisfying $\|Z\|_{L^\infty(\Omega,\mathbb{R}^N)}\le 1$ and $(Z, Du)=|Du|$ as measures (see also \cite{ACM04}). For the inhomogeneous case, the nonlinear semigroup theory was also utilized in \cite{T11} when $f\in L^2(\Omega_T)$. Under the assumption that the source $f$ belongs to $L^1_{\text{loc}}(0,+\infty, L^2(\Omega))$, Latorre and Segura de Le\'{o}n demonstrated the existence and uniqueness of solutions in \cite{LS22}, which retrieved the results of \cite{SW15} in a correct way. Additional results concerning problem \eqref{1.2} can be found in \cite{BDM15, BDS16, BDSS19}.

Regarding the nonlocal case, it is worth mentioning that the definition of solutions to \eqref{1.1} was originally given by Maz\'{o}n, Rossi and Toledo in \cite{MRT16}, where they showed that a suitable function $Z\in L^\infty\left( \mathbb{R}^N\times \mathbb{R}^N\times [0, T]\right) $ should be found to interpret and describe the fractional ratio $\frac{u(x,t)-u(y,t)}{|u(x,t)-u(y,t)|}$ and studied the existence of solutions to homogeneous parabolic fractional 1-Laplacian equation. Different from the local case, the function $Z$ describes $\frac{u(x,t)-u(y,t)}{|u(x,t)-u(y,t)|}$ in the set $\left\lbrace \mathbb{R}^N\times\mathbb{R}^N\times[0,T]\big|u(x,t)\neq u(y,t)\right\rbrace $ precisely and $Z\in [-1,1]$ in the set $\left\lbrace \mathbb{R}^N\times\mathbb{R}^N\times[0,T]\big|u(x,t)= u(y,t)\right\rbrace $. In terms of elliptic equations, the existence of $(s,1)$-harmonic function was studied in \cite{BU21}. In addition, by imposing the condition that $f$ belongs to $L^{\frac{N}{s}}(\Omega)$ and its norm is sufficiently small, Bucur investigated in \cite{BU23} the minimizer of the following functional
\begin{align*}
	\mathcal{F}^s_1(u):=\frac{1}{2}\int_{\mathbb{R}^N\times\mathbb{R}^N}\frac{|u(x)-u(y)|}{|x-y|^{N+s}}dxdy-\int_{\Omega}fudx
\end{align*}
and showed that this minimizer corresponds to the weak solution of the nonlocal 1-Laplacian problem. On the other hand, when $\|f\|_{L^\frac{N}{s}(\Omega)}>\frac{1}{S_{N,s}}$ the energy could be unbounded and the minimizer may not exist in this case. What's more, the flatness results for the minimizer were obtained, which implies that the expression of $Z$ towards $\frac{u(x)-u(y)}{|u(x)-u(y)|}$ is not precise. We also refer to \cite{NO23} for an interesting result, in which Novage and Onoue studied the minimizers of a nonlocal variational problem concerning the image-denoising model
\begin{align*}
	\text{min}\left\lbrace \mathcal{F}_{K,f}(u)\big|u\in BV_K(\mathbb{R}^N)\cap L^2(\mathbb{R}^N)\right\rbrace ,
\end{align*}
where $\mathcal{F}_{K,f}$ is given by
\begin{align*}
	\mathcal{F}_{K,f}(u):=\frac{1}{2}\int_{\mathbb{R}^N\times\mathbb{R}^N}K(|x-y|)|u(x)-u(y)|dxdy+\frac{1}{2}\int_{\mathbb{R}^N}\left( u(x)-f(x)\right)^2dx 
\end{align*}
with $K$ being a kernel singular at the origin. A typical example is that $K(x)=\frac{1}{|x|^{N+s}}$. In two dimensions, by assuming the original image is locally $\beta$-H\"{o}lder continuous with $\beta\in (1-s,1]$, they showed the minimizers have the same local H\"{o}lder regularity.

Finally, we would like to turn to the method of Rothe time-discretization, also called the minimizing movements method or difference method, which is usually used to prove the existence of solutions to parabolic problems. We refer to \cite{WZ06, XZ07, CZ09, ZZ10} for the application in terms of showing the existence of weak solutions of nonlinear parabolic problems and refer to \cite{SS23} for the application in terms of finding variational solutions.

Motivated by the results mentioned above, our main aim in the present paper is to study the existence and uniqueness of weak solutions to the nonlocal parabolic 1-Laplacian problem \eqref{1.1}. Before we begin our proof of the main result, it is crucial to demonstrate the equivalence between the minimizers and the weak solutions to the following problem
\begin{align}
	\label{1.3}
	\begin{cases}
		\dfrac{u-u_0}{h}+\left( -\Delta\right)^s_1u=f &\text{in }\Omega,\\
		u=0 &\text{in }\mathcal{C}_\Omega,
	\end{cases}
\end{align}
where $f\in L^2(\Omega)$. However, when dealing with problem \eqref{1.1}, it is necessary to assume $f\in L^\frac{N}{s}(\Omega_T)$ to show that the weak solution belongs to $L^\infty(0,T;L^2(\Omega))$. Similar to the method of difference, a function $u$ could be found and proved to be the solution of \eqref{1.1} and a function $Z$ may be found to replace $\frac{u(x,t)-u(x,y)}{|u(x,t)-u(y,t)|}$ in the weak formulation.

From the purely mathematical point of view, the super-linear growth condition on the operator seems to be of great significance because a convergence result is indispensable (see \cite[Lemma 2.5]{ZZ10} and \cite[Lemma 2.6]{CZ09}). In the context of the $1$-Laplacian equation, diverging from the classic approach, we have to use an alternative method to derive the convergence result concerning $u$, provided that the initial value $u_0$ possesses finite energy. In other words, $u_0$ not only should belong to $L^2(\Omega)$, but also should lie in the fractional Sobolev space $W^{s,1}_0(\Omega)$. Additionally, it should be emphasized that based on the uniform boundedness result, one can expect $u$ to be $\frac{1}{2}$-H\"{o}lder continuous in time.

Consider the set-value sign function $\text{sgn}$ given by
\begin{align*}
	\text{sgn}(s)=
	\begin{cases}
		1&\text{if }s>0,\\
		[-1,1]&\text{if }s=0,\\
		-1&\text{if }s<0.
	\end{cases}
\end{align*}
We introduce the weak solution to such a singular parabolic problem.
\begin{definition}
	\label{def1.1}
	A function $u\in C^{0,\frac{1}{2}}\left( [0,T]; L^2(\Omega)\right) \cap L^\infty( 0,T;W^{s,1}_0(\Omega))  $ is said to be a weak solution to problem \eqref{1.1} if there exists a function $Z\in L^\infty(\mathbb{R}^N\times\mathbb{R}^N\times[0,T])$ satisfying:
	\begin{itemize}
		\item[(1)] For all $\varphi\in C^1([0, T]; L^2(\Omega))\cap C([0, T]; W^{s, 1}_0(\Omega))$ with $\varphi(\cdot, T)=0$, there holds
		\begin{align}
			\label{1.4}
			\begin{aligned}
				&\quad-\int_{\Omega}u_0\varphi(x, 0)dx-\int_{0}^{T}\int_{\Omega}u\varphi_tdxdt+\int_{0}^{T}\int_{\mathbb{R}^N\times\mathbb{R}^N}Z\frac{\varphi(x,t)-\varphi(y,t)}{|x-y|^{N+s}}dxdydt\\
				&=\int_{0}^{T}\int_{\Omega}f\varphi dxdt;
			\end{aligned}
		\end{align}
		\item[(2)] $Z(x,y,t)\in \mathrm{sgn}(u(x, t)-u(y, t))$ is an anti-symmetry function with respect to $(x,y)$. 
	\end{itemize}
\end{definition}
\begin{remark}
	\label{rem1.2}
	By an approximation argument, the weak solution $u$ could be chosen as a test function in \eqref{1.4}.
\end{remark}

We are now in a position to state the main result as follows.

\begin{theorem}
	\label{the1.1}
	Under the assumption that $f\in L^\frac{N}{s}(\Omega_T)$ and $u_0\in L^2(\Omega)\cap W^{s,1}_0(\Omega)$, there exists a weak solution to problem \eqref{1.1} in the sense of Definition \ref{def1.1}.
\end{theorem}
This paper is organized as follows. In Section \ref{sec2}, we state some notations regarding the fractional Sobolev spaces and give some crucial lemmas. Section \ref{sec3} is devoted to proving that the minimizers of a functional are equivalent to the weak solutions to problem \eqref{1.3}. We will prove our main result in Section \ref{sec4}.

\section{Preliminaries}
\label{sec2}
In this section, we introduce some notations that will be used later. We say that $\left\lbrace u_p\right\rbrace $ is a sequence and consider the subsequence of it. We denote by $|z|$ the Euclidean norm of $z\in \mathbb{R}^N$. If not otherwise specified, we denote by $C$ several constants whose value may change from line to line even if in the same inequality.

For a function $f\in L^q(\Omega_T)$ with $q\in[1, +\infty)$, we denote by 
\begin{align*}
	[f]_h(x, t):=\frac{1}{h}\int_{t}^{t+h}f(x, \mu)d\mu
\end{align*}
the Steklov average of $f$ and we have the following conclusion.
\begin{lemma}
	\label{lem2.0}
	Suppose a function $f$ belongs to $L^q(\Omega_T)$ with $q\in [1,+\infty)$, there holds
	\begin{align*}
		[f]_h\rightarrow f \quad \text{in } L^q(\Omega_T)\quad(h\rightarrow0).
	\end{align*}
\end{lemma}

For $0<s<1$ and $1\le p<+\infty$, the fractional Sobolev space $W^{s,p}(\mathbb{R}^N)$ is defined as
\begin{align*}
	\left\lbrace 
	u\in L^p(\mathbb{R}^N)\bigg|[u]_{W^{s,p}(\mathbb{R}^N)}:=\left( \int_{\mathbb{R}^N\times\mathbb{R}^N}\frac{|u(x)-u(y)|^p}{|x-y|^{N+sp}}dxdy\right)^\frac{1}{p}<+\infty 
	\right\rbrace .
\end{align*}
For a smooth bounded domain $\Omega\subset\mathbb{R}^N$, we denote by $W^{s,p}_0(\Omega)$ a fractional Sobolev space, which is
\begin{align*}
	W^{s,p}_0(\Omega):=
	\left\lbrace 
	u\in W^{s,p}(\mathbb{R}^N)\big|u=0\text{ in }\mathbb{R}^N\backslash\Omega 
	\right\rbrace.
\end{align*}
We refer to \cite{DPV12} for more information about those fractional Sobolev spaces and emphasize that the space $L^2(\Omega)\cap W^{s,p}_0(\Omega)$ is a closed and convex set.

Under the fractional framework, we define the following functionals
\begin{align*}
	\Phi^s_p(u):=\frac{1}{2p}\int_{\mathbb{R}^N\times\mathbb{R}^N}\frac{|u(x)-u(y)|^p}{|x-y|^{N+sp}}dxdy,
\end{align*}
and 
\begin{align*}
	J^s_p(u):=\frac{2}{hp}\int_{\mathbb{R}^N\times\mathbb{R}^N}\frac{|u(x)-u(y)|^p}{|x-y|^{N+sp}}dxdy+\int_{\Omega}\left( \frac{u-u_0}{h}-f\right)^2dx, 
\end{align*}
where $0<s<1$, $1\le p<+\infty$, and $h>0$ are fixed constants.

To begin with, we prove the existence and uniqueness of weak solutions to the following elliptic problem
\begin{align}
	\label{2.1}
	\begin{cases}
		\dfrac{u-u_0}{h}+\left( -\Delta\right)^s_pu=f &\text{in }\Omega,\\
		u=0 &\text{in }\mathcal{C}_\Omega,
	\end{cases}
\end{align}
where $u_0, f\in L^2(\Omega)$, and $h>0, p>1$ are fixed constants.
\begin{definition}
	\label{def2.1}
	A function $u\in L^2(\Omega)\cap W^{s,p}_0(\Omega)$ is said to be a weak solution to problem \eqref{2.1} if for every $\varphi\in L^2(\Omega)\cap W^{s,p}_0(\Omega)$, there holds
	\begin{align}
		\label{2.2}
		\begin{aligned}
			&\quad\int_{\Omega}\frac{u-u_0}{h}\varphi dx+\int_{\mathbb{R}^N\times\mathbb{R}^N}\frac{|u(x)-u(y)|^{p-2}(u(x)-u(y))(\varphi(x)-\varphi(y))}{|x-y|^{N+sp}}dxdy\\
			&=\int_{\Omega}f\varphi dx.
		\end{aligned}
	\end{align}
\end{definition}
\begin{lemma}
	\label{lem2.1}
	There exists a unique weak solution $u\in L^2(\Omega)\cap W^{s,p}_0(\Omega)$ to problem \eqref{2.1}.
\end{lemma}
\begin{proof}
Consider the variational problem
	\begin{align*}
		\min\left\lbrace J^s_p(v)\big|v\in L^2(\Omega)\cap W^{s,p}_0(\Omega)\right\rbrace.
	\end{align*}
	We will show that $J^s_p(v)$ has a minimizer $u$ in $L^2(\Omega)\cap W^{s,p}_0(\Omega)$ and then prove that this minimizer is a weak solution to problem \eqref{2.1}.
	
	We first show that $J^s_p$ is coercive on $L^2(\Omega)\cap W^{s,p}_0(\Omega)$. For any $v\in L^2(\Omega)\cap W^{s,p}_0(\Omega)$,
	\begin{align*}
		J^s_p(v)&:=\frac{2}{hp}\int_{\mathbb{R}^N\times\mathbb{R}^N}\frac{|v(x)-v(y)|^p}{|x-y|^{N+sp}}dxdy+\int_{\Omega}\left( \frac{v-u_0}{h}-f\right)^2dx\\
		&\ \ge \frac{2}{hp}[v]^p_{W^{s,p}_0(\Omega)}+\left( \bigg\|\frac{v}{h}\bigg\|_{L^2(\Omega)}-\bigg\|\frac{u_0}{h}\bigg\|_{L^2(\Omega)}-\|f\|_{L^2(\Omega)} \right)^2,
	\end{align*}
	by letting $\|v\|_{L^2(\Omega)\cap W^{s,p}_0(\Omega)}$ goes to infinity, we have $J^s_p(v)\rightarrow+\infty$.
	
Next, we will show the existence of minimizers of $J^s_p$. Because of
	\begin{align*}
		0\le \inf_{v\in L^2(\Omega)\cap W^{s,p}_0(\Omega)}J^s_p(v)\le J^s_p(0)=\int_\Omega\left( f+\frac{u_0}{h}\right)^2dx,  
	\end{align*}
	we could find a  minimizing sequence $\left\lbrace v_m\right\rbrace \subset L^2(\Omega)\cap W^{s,p}_0(\Omega)$, satisfying
	\begin{align*}
		\lim\limits_{m\rightarrow+\infty}J^s_p(v_m)=\inf_{v\in L^2(\Omega)\cap W^{s,p}_0(\Omega)}J^s_p(v).
	\end{align*}
	Thus there exists a constant $C$ such that
	\begin{align*}
		\frac{2}{hp}\int_{\mathbb{R}^N\times\mathbb{R}^N}\frac{|v_m(x)-v_m(y)|^p}{|x-y|^{N+sp}}dxdy+\int_\Omega\left( \frac{v_m-u_0}{h}-f\right)^2dx\le C,
	\end{align*}
	which implies that
	\begin{align*}
		\|v_m\|_{W^{s,p}_0(\Omega)}\le C[u]_{W^{s,p}_0(\Omega)}\le 	C(f,u_0)
	\end{align*}
	and
	\begin{align*}
		\|v_m\|_{L^2(\Omega)}\le C(f,u_0).
	\end{align*}
	Therefore, there exists a function $u\in L^2(\Omega)\cap W^{s,p}_0(\Omega)$ to which $\left\lbrace u_m\right\rbrace $, up to a subsequence,
	\begin{align*}
		v_m\rightharpoonup u \quad\text{weakly in }L^2(\Omega)\cap W^{s,p}_0(\Omega).
	\end{align*}
	Combining with the fact that $J^s_p$ is weakly lower semi-continuous on $L^2(\Omega)\cap W^{s,p}_0(\Omega)$, we obtain
	\begin{align*}
		\inf_{v\in L^2(\Omega)\cap W^{s,p}_0(\Omega)}J^s_p(v)\le J^s_p(u)\le \liminf_{m\rightarrow+\infty}J^s_p(v_m) =\inf_{v\in L^2(\Omega)\cap W^{s,p}_0(\Omega)}J^s_p(v).
	\end{align*}
	It follows that $u\in L^2(\Omega)\cap W^{s,p}_0(\Omega)$ is a minimizer of the functional $J^s_p$.
	
	To see that the minimizer is a weak solution to problem \eqref{2.1}, we may perturb $u$ with test function $\varphi\in L^2(\Omega)\cap W^{s,p}_0(\Omega)$ and use that the first variation of the energy vanishes to deduce $u$ is a weak solution, i.e.,
	\begin{align*}
		0&=\frac{d}{dt}J^s_p(u+t\varphi)\bigg|_{t=0}\\
		&=\frac{d}{dt}\left[ \frac{2}{hp}\int_{\mathbb{R}^N\times\mathbb{R}^N}\frac{|u(x)+t\varphi(x)-u(y)-t\varphi(y)|^p}{|x-y|^{N+sp}}dxdy+\int_{\Omega}\left( \frac{u+t\varphi-u_0}{h}-f\right)^2dx\right] \Bigg|_{t=0}\\
		&=\frac{2}{h}\int_{\mathbb{R}^N\times\mathbb{R}^N}\frac{|u(x)-u(y)|^{p-2}(u(x)-u(y))(\varphi(x)-\varphi(y))}{|x-y|^{N+sp}}dxdy\\
		&\quad+\frac{2}{h}\int_{\Omega}\left( \frac{u-u_0}{h}-f\right)\varphi dx .
	\end{align*}
	
	Finally, suppose that there is a function $\tilde{u}\neq u$, belonging to $\in L^2(\Omega)\cap W^{s,p}_0(\Omega)$ and satisfying the weak formulation \eqref{2.2}, we have
	\begin{align}
		\label{eq1}
		\int_{\Omega}\frac{(u-\tilde{u})^2}{h}dx+\int_{\mathbb{R}^N\times\mathbb{R}^N}\frac{U(x,y)}{|x-y|^{N+sp}}dxdy=0,
	\end{align}
	where $U(x,y)$ is defined as follows
	\begin{align*}
		U(x,y):=&\left( |u(x)-u(y)|^{p-2}(u(x)-u(y)-|\tilde{u}(x)-\tilde{u}(y)|^{p-2}(\tilde{u}(x)-\tilde{u}(y)))\right) \\
		&\times(u(x)-\tilde{u}(x)-u(y)+\tilde{u}(y)).
	\end{align*}
	Since the two terms on the left-hand side of \eqref{eq1} are nonnegative, we have $u=\tilde{u}$ a.e. in $\Omega$. Therefore we obtain the uniqueness of weak solutions.
\end{proof}

Next, we list two lemmas regarding the functional $\Phi^s_p$ and fractional Sobolev space. The proofs can be found in  \cite{BU23}.
\begin{lemma}
	\label{lem2.2}
	For fixed $0<s<1$, set $s_p:=N+s-\frac{N}{p}\in (s,1)$, then for all $u\in W^{s_p,p}_0(\Omega)$, there holds that
	\begin{align*}
		[u]_{W^{s,1}_0(\Omega)}\le C^\frac{1-p}{p}[u]_{W^{s_p,p}_0(\Omega)},
	\end{align*}
	where $C$ depends on $N,s,\Omega$.
\end{lemma}
\begin{lemma}
	\label{lem2.3}
	For fixed $0<s<1$ and $1<q<\min\left\lbrace \frac{N}{N+s-1},\frac{N+1}{N+s}\right\rbrace $, set $s_q:=N+s-\frac{N}{q}\in (s,1)$, if $u\in W^{s_q,q}_0(\Omega)$, it holds that
	\begin{align*}
		\lim\limits_{p\rightarrow1}\Phi^{s_p}_p(u)=\Phi^s_1(u).
	\end{align*}
\end{lemma}

\section{Results on elliptic equations}
\label{sec3}
In this section, we address problem \eqref{1.3}, that is,
\begin{align*}
	\begin{cases}
		\dfrac{u-u_0}{h}+\left( -\Delta\right)^s_1u=f &\text{in }\Omega,\\
		u=0 &\text{in }\mathcal{C}_\Omega,
	\end{cases}
\end{align*}
demonstrating the existence of a minimizers for $J^s_1$ on $L^2(\Omega)\cap W^{s,1}_0(\Omega)$ and establishing the equivalence between weak solutions and minimizers.

\begin{definition}
	\label{def2.2}
	A function $u\in L^2(\Omega)\cap W^{s,1}_0(\Omega)$ is said to be a weak solution to problem \eqref{1.3} if there exists a function $Z\in L^\infty(\mathbb{R}^N\times\mathbb{R}^N)$, satisfying
	\begin{itemize}
		\item[(1)] For every $\varphi\in L^2(\Omega)\cap W^{s,1}_0(\Omega)$, there holds
		\begin{align*}
			\int_{\Omega}\frac{u-u_0}{h}\varphi dx+\int_{\mathbb{R}^N\times\mathbb{R}^N}Z\frac{\varphi(x)-\varphi(y)}{|x-y|^{N+s}}dxdy=\int_{\Omega}f\varphi dx;
		\end{align*}
		\item[(2)] $Z(x,y)\in \mathrm{sgn}(u(x)-u(y))$ is an anti-symmetry function with respect to $(x,y)$.
	\end{itemize}
\end{definition}
\begin{lemma}
	\label{lem2.4}
	For fixed $s\in (0,1)$, set $s_p:=N+s-\frac{N}{p}$. Let $u_p$ be the minimizer of the functional $J^{s_p}_p$ on $L^2(\Omega)\cap W^{s_p,p}_0(\Omega)$, then there exists a function $u\in L^2(\Omega)\cap W^{s,1}_0(\Omega)$ such that, up to a subsequence,
	\begin{align*}
		u_p\rightarrow u \quad\text{strongly in }L^q(\Omega) \text{ and a.e. in }\mathbb{R}^N,
	\end{align*}
	where $1\le q<2$. Moreover, the function $u$ is the unique minimizer of the functional $J^s_1$ in $L^2(\Omega)\cap W^{s,1}_0(\Omega)$.
\end{lemma}
\begin{proof}
	We prove the consequence by several steps. Some of the reasoning is based on the ideas developed in \cite{BU23}.
	
\textbf{Step 1. }To prove the boundedness of $\left\lbrace u_p\right\rbrace $ in $L^2(\Omega)\cap W^{s,1}_0(\Omega)$.
	
	For fixed $p>1$, by Lemma \ref{lem2.2} and Definition \ref{def2.1}, we have
	\begin{align*}
		[u_p]^p_{W^{s,1}_0(\Omega)}&\le C^{1-p}(N,s,\Omega)[u_p]^p_{W^{s_p,p}(\Omega)}\\
		&=C^{1-p}(N,s,\Omega)\left( \int_{\Omega}fu_pdx-\int_{\Omega}\frac{u_p-u_0}{h}u_pdx\right)\\
		&\le C^{1-p}(N,s,\Omega) \int_{\Omega}\left( \frac{f^2h}{2}+\frac{u_p^2}{2h}-\frac{u_p^2}{h}+\frac{u_0^2}{2h}+\frac{u_p^2}{2h}\right)dx\\
		&= \frac{1}{2} C^{1-p}(N,s,\Omega)\int_{\Omega}\left( f^2h+\frac{u_0^2}{h}\right)dx\\
		&\le \left( \frac{C(N,s,\Omega)}{C(N,s,\Omega)+1}\right)^{{1-p}} \int_{\Omega}\left( f^2h+\frac{u_0^2}{h}\right)dx\\
		&\le \left( 1+\frac{1}{C(N,s,\Omega)}\right)\int_{\Omega}\left( f^2h+\frac{u_0^2}{h}\right)dx,
	\end{align*}
	where the Young's inequality was utilized and $p$ was limited in $(1,2)$.
	
	In addition, by taking $u_p$ as test function in \eqref{2.2} and observing the following inequality
	\begin{align*}
		\int_{\Omega}u_p^2dx&\le h\int_{\Omega}fu_pdx+\int_{\Omega}u_0u_pdx\\
		&\le \frac{1}{2}h^2\int_{\Omega}f^2dx+\frac{1}{2}\int_{\Omega}u_p^2dx+2\int_{\Omega}u_0^2dx+\frac{1}{4}\int_{\Omega}u_p^2dx,
	\end{align*}
	we know that the sequence $\left\lbrace u_p\right\rbrace $ is bounded in $L^2(\Omega)\cap W^{s,1}_0(\Omega)$. Thus, by compactness results involving the fractional Sobolev spaces, up to subsequences,
\begin{align*}
u_p\to u  \quad \text{a.e. in } \mathbb{R}^N
\end{align*}
and 
\begin{align*}
u_p\rightarrow u \quad\text{strongly in } L^q(\Omega)  \quad \text{for } 1 \le q<2.
\end{align*}

	\textbf{Step 2. }To prove that $u$ found in Step 1 is a minimizer of $J^s_1$.
	
	Suppose $u$ is not the minimizer of $J^s_1$, we may find a function $v$, such that
	\begin{align}
		\label{2.3}
		J^s_1(v)<J^s_1(u).
	\end{align}
	For reason of $C^\infty_0(\Omega)$ is dense in $L^2(\Omega)\cap W^{s,1}_0(\Omega)$, we know that there exists a sequence $\left\lbrace v_m\right\rbrace \subset C^\infty_0(\Omega)$, such that
	\begin{align*}
		v_m\rightarrow v\quad\text{in }L^2(\Omega)\cap W^{s,1}_0(\Omega)\text{ and a.e. in }\mathbb{R}^N.
	\end{align*}
	Hence
	\begin{align*}
		\lim\limits_{m\rightarrow+\infty}J^s_1(v_m)&=\lim\limits_{m\rightarrow+\infty}\left[ \frac{2}{h}\int_{\mathbb{R}^N\times\mathbb{R}^N}\frac{|v_m(x)-v_m(y)|}{|x-y|^{N+s}}dxdy+\int_{\Omega}\left( \frac{v_m-u_0}{h}-f\right)^2dx \right]\\
		&= \frac{2}{h}\int_{\mathbb{R}^N\times\mathbb{R}^N}\frac{|v(x)-v(y)|}{|x-y|^{N+s}}dxdy+\int_{\Omega}\left( \frac{v-u_0}{h}-f\right)^2dx\\
		&=J^s_1(v).
	\end{align*}
	According to Lemma \ref{lem2.3}, Fatou's Lemma and the minimality of $u_p$, we derive that
	\begin{align*}
		J^s_1(u)&\le \liminf\limits_{p\rightarrow1}J^{s_p}_p(u_p)\le \limsup\limits_{p\rightarrow1}J^{s_p}_p(u_p)\le \limsup\limits_{p\rightarrow1}J^{s_p}_p(v_m)\\
		&=\limsup\limits_{p\rightarrow1}\left[ \frac{4}{h}\Phi^s_p(v_m)+\int_{\Omega}\left( \frac{v_m-u_0}{h}-f\right)^2dx \right]\\
		&=\frac{4}{h}\Phi^s_1(v_m)+\int_{\Omega}\left( \frac{v_m-u_0}{h}-f\right)^2dx\\
		&=J^s_1(v_m).
	\end{align*}
	By letting $m$ goes to $+\infty$, we have
	\begin{align*}
		J^s_1(u)\le J^s_1(v),
	\end{align*}
	which is contradict to \eqref{2.3}.
	
\textbf{Step 3. }To prove the uniqueness of minimizers.
	
	Suppose there exists a function $\tilde{u}\neq u$, satisfying for any $v\in L^2(\Omega)\cap W^{s,1}_0(\Omega)$, $J^s_1(u)=J^s_1(\tilde{u})\le J^s_1(v)$. Thus, by setting $\bar{u}=\frac{1}{2}u+\frac{1}{2}\tilde{u}$, we have
	\begin{align*}
		J^s_1(\bar{u})&=\frac{2}{h}\int_{\mathbb{R}^N\times\mathbb{R}^N}\frac{|\bar{u}(x)-\bar{u}(y)|}{|x-y|^{N+s}}dxdy+\int_{\Omega}\left( \frac{\bar{u}-u_0}{h}-f\right)^2dx\\
		&=\frac{2}{h}\int_{\mathbb{R}^N\times\mathbb{R}^N}\frac{|\frac{1}{2}u(x)+\frac{1}{2}\tilde{u}(x)-\frac{1}{2}u(y)-\frac{1}{2}\tilde{u}(y)|}{|x-y|^{N+s}}dxdy
		+\int_{\Omega}\left( \frac{\frac{1}{2}u+\frac{1}{2}\tilde{u}-u_0}{h}-f\right)^2dx\\
		&< \frac{1}{h}\int_{\mathbb{R}^N\times\mathbb{R}^N}\frac{|u(x)-u(y)|+|\tilde{u}(x)-\tilde{u}(y)|}{|x-y|^{N+s}}dxdy\\
		&\quad+\frac{1}{2}\int_{\Omega}\left( \frac{u-u_0}{h}-f\right)^2dx+\frac{1}{2}\int_{\Omega}\left( \frac{\tilde{u}-u_0}{h}-f\right)^2dx\\
		&=\frac{1}{2}J^s_1(u)+\frac{1}{2}J^s_1(\tilde{u})\\
		&=J^s_1(u),
	\end{align*}
	which is contradict to $u$ being a minimizer of $J^s_1$.
\end{proof}

In the rest part of this section, we shall show the equivalence between the minimizer of $J^s_1$ and the weak solution to problem \eqref{1.3}.
\begin{lemma}
	\label{lem2.5}
	For fixed $s\in (0,1)$, suppose that there exists a weak solution to problem \eqref{1.3}, then the function $u$ found in Lemma \ref{lem2.4} is a weak solution to \eqref{1.3}. 
\end{lemma}
\begin{proof}
	Let $\tilde{u}$ be a weak solution to \eqref{1.3}. From Definition \ref{def2.2}, we know that there exists a function $Z$ such that
	\begin{itemize}
		\item[(1)] For every $\varphi\in L^2(\Omega)\cap W^{s,1}_0(\Omega)$,
		\begin{align*}
			\int_{\Omega}\frac{\tilde{u}-u_0}{h}\varphi dx+\int_{\mathbb{R}^N\times\mathbb{R}^N}Z\frac{\varphi(x)-\varphi(y)}{|x-y|^{N+s}}dxdy=\int_{\Omega}f\varphi dx;
		\end{align*}
		\item[(2)] $Z(x,y)\in \mathrm{sgn}(\tilde{u}(x)-\tilde{u}(y))$ is an anti-symmetry function with respect to $(x,y)$.
	\end{itemize}
	Choosing $\varphi=\tilde{u}-u$ and noting that
	\begin{align*}
		&\quad\int_{\Omega}\left[ \frac{\tilde{u}-u_0}{h}(\tilde{u}-u)-f(\tilde{u}-u)\right]dx\\
		&=\int_{\Omega}\left( \frac{\tilde{u}-u_0}{h}-f\right)\left( \tilde{u}-u_0-hf-u+u_0+hf\right)dx\\
		&=h\int_{\Omega}\left[ \left( \frac{\tilde{u}-u_0}{h}-f\right)^2-\left( \frac{\tilde{u}-u_0}{h}-f\right)\left( \frac{u-u_0}{h}-f\right)   \right]dx\\
		&\ge\frac{h}{2}\int_{\Omega}\left[ \left( \frac{\tilde{u}-u_0}{h}-f\right)^2-\left( \frac{u-u_0}{h}-f\right)^2  \right]dx, 
	\end{align*}
	we have
	\begin{align*}
		J^s_1(u)&=\frac{2}{h}\int_{\mathbb{R}^N\times\mathbb{R}^N}\frac{|u(x)-u(y)|}{|x-y|^{N+s}}dxdy+\int_{\Omega}\left( \frac{u-u_0}{h}-f\right)^2dx\\
		&\ge \frac{2}{h}\int_{\mathbb{R}^N\times\mathbb{R}^N}Z\frac{u(x)-u(y)}{|x-y|^{N+s}}dxdy+\int_{\Omega}\left( \frac{u-u_0}{h}-f\right)^2dx\\
		&\ge \frac{2}{h}\int_{\mathbb{R}^N\times\mathbb{R}^N}Z\frac{u(x)-u(y)}{|x-y|^{N+s}}dxdy+\int_{\Omega}\left( \frac{\tilde{u}-u_0}{h}-f\right)^2dx\\
		&\quad-\frac{2}{h}\int_{\Omega}\left[ \frac{\tilde{u}-u_0}{h}(\tilde{u}-u)-f(\tilde{u}-u)\right]dx\\
		&= \int_{\Omega}\left( \frac{\tilde{u}-u_0}{h}-f\right)^2dx-\frac{2}{h}\int_{\Omega}\left( \frac{\tilde{u}-u_0}{h}\tilde{u}-f\tilde{u}\right) dx\\
		&= \frac{2}{h}\int_{\mathbb{R}^N\times\mathbb{R}^N}\frac{|\tilde{u}(x)-\tilde{u}(y)|}{|x-y|^{N+s}}dxdy+\int_{\Omega}\left( \frac{\tilde{u}-u_0}{h}-f\right)^2dx \\
		&=J^s_1(\tilde{u}).
	\end{align*}
	Therefore, $u=\tilde{u}$ and the conclusion follows.
\end{proof}
\begin{remark}
	\label{rem2.1}
	{\rm (1)} We ignored the proof of the existence of weak solutions and left the details to \cite[Theorem 3.4]{MRT16} and \cite[Theorem 1.6]{BU23}.
	
	{\rm (2)} The equivalence between the minimizers of $J^s_1$ and the weak solutions to problem \eqref{1.3} could be proved easily by combining Lemmas \ref{lem2.4} and \ref{lem2.5}.
\end{remark}

\section{Proof of the main result}
\label{sec4}
In this section, we will prove the existence and uniqueness of weak solutions to problem \eqref{1.1}. 

\begin{proof}[Proof of Theorem \ref{the1.1}]
	We divide our proof into several parts. 
	
	First, we shall show the asymptotic behaviour of $\left\lbrace u^m\right\rbrace $ and $\left\lbrace Z^m\right\rbrace $, which are defined later. For a fixed number $m\in \mathbb{N}$, consider the following elliptic problem
	\begin{align*}
		\begin{cases}
			\displaystyle
			\frac{u_k-u_{k-1}}{h}+(-\Delta)^s_1u_k=[f]_{h}(x,(k-1)h)&\text{in }\Omega,\\
			u_k=0&\text{in }\mathcal{C}_\Omega,
		\end{cases}
	\end{align*}
	where $h:=\frac{T}{m}$, $k\in\left\lbrace 1, 2, \dots, m\right\rbrace $, and $[f]_h$ is the steklov average of $f$.
	
	By Lemma \ref{lem2.5}, for every $k$, there exists a weak solution $u_k\in L^2(\Omega)\cap W^{s,1}_0(\Omega)$ and a function $Z_k\in \text{sgn}(u_k(x)-u_k(y))$, such that for every $\varphi\in L^2(\Omega)\cap W^{s,1}_0(\Omega)$,
	\begin{align}
		\label{4.1}
		\int_{\Omega}\frac{u_k-u_{k-1}}{h}\varphi dx+\int_{\mathbb{R}^N\times\mathbb{R}^N}Z_k\frac{\varphi(x)-\varphi(y)}{|x-y|^{N+s}}dxdy=\int_{\Omega}[f]_h(x,(k-1)h)\varphi dx.
	\end{align}
	We take $\varphi=u_k-u_{k-1}$ as test function in \eqref{4.1} and use the Young's inequality to deduce a estimate for $u_k$, i.e.,
	\begin{align*}
		&\quad\int_{\Omega}\frac{\left( u_k-u_{k-1}\right) ^2}{h} dx+\int_{\mathbb{R}^N\times\mathbb{R}^N}\frac{|u_k(x)-u_k(y)|}{|x-y|^{N+s}}dxdy\\
		&=\int_{\Omega}[f]_h(x,(k-1)h)(u_k-u_{k-1}) dx+\int_{\mathbb{R}^N\times\mathbb{R}^N}Z_k\frac{u_{k-1}(x)-u_{k-1}(y)}{|x-y|^{N+s}}dxdy\\
		&\le \frac{h}{4}\int_{\Omega}[f]_h^2(x,(k-1)h)dx+\int_{\Omega}\frac{\left( u_k-u_{k-1}\right) ^2}{h}dx+\int_{\mathbb{R}^N\times\mathbb{R}^N}Z_k\frac{u_{k-1}(x)-u_{k-1}(y)}{|x-y|^{N+s}}dxdy\\
		&\le \frac{h}{4}\int_{\Omega}[f]_h^2(x,(k-1)h)dx+\int_{\Omega}\frac{\left( u_k-u_{k-1}\right) ^2}{h}dx+\int_{\mathbb{R}^N\times\mathbb{R}^N}\frac{|u_{k-1}(x)-u_{k-1}(y)|}{|x-y|^{N+s}}dxdy,
	\end{align*}
	which implies that
	\begin{align}
		\label{4.2}
		\begin{aligned}
			&\quad \int_{\mathbb{R}^N\times\mathbb{R}^N}\frac{|u_k(x)-u_k(y)|}{|x-y|^{N+s}}dxdy\\
			&\le \frac{h}{4}\int_{\Omega}[f]_h^2(x,(k-1)h)dx+\int_{\mathbb{R}^N\times\mathbb{R}^N}\frac{|u_{k-1}(x)-u_{k-1}(y)|}{|x-y|^{N+s}}dxdy\\
			&= \frac{h}{4}\int_{\Omega}\left( \frac{1}{h}\int_{(k-1)h}^{kh}f(x,t)dt\right)^2 dx+\int_{\mathbb{R}^N\times\mathbb{R}^N}\frac{|u_{k-1}(x)-u_{k-1}(y)|}{|x-y|^{N+s}}dxdy\\
			&\le \frac{1}{4h}\left[ \int_{(k-1)h}^{kh}\left( \int_{\Omega}f^2(x,t)dx\right)^\frac{1}{2} dt\right] ^2+\int_{\mathbb{R}^N\times\mathbb{R}^N}\frac{|u_{k-1}(x)-u_{k-1}(y)|}{|x-y|^{N+s}}dxdy\\
			&\le \frac{1}{4} \int_{(k-1)h}^{kh}\left( \int_{\Omega}f^2(x,t)dx\right) dt+\int_{\mathbb{R}^N\times\mathbb{R}^N}\frac{|u_{k-1}(x)-u_{k-1}(y)|}{|x-y|^{N+s}}dxdy.\\
		\end{aligned}
	\end{align}
	We add all inequalities \eqref{4.2} for $k=1, 2, \dots, i$ to get
	\begin{align}
		\label{4.3}
		\int_{\mathbb{R}^N\times\mathbb{R}^N}\frac{|u_i(x)-u_i(y)|}{|x-y|^{N+s}}dxdy\le \int_{\mathbb{R}^N\times\mathbb{R}^N}\frac{|u_0(x)-u_0(y)|}{|x-y|^{N+s}}dxdy+C,
	\end{align}
	where $C$ depends on $f$.
	
	Similarly, by taking $\varphi=u_k$ in \eqref{4.1}, we have
	\begin{align*}
		&\quad\int_{\Omega}\frac{u_k-u_{k-1}}{h}u_kdx\\
		&\le \int_{\Omega}[f]_h(x,(k-1)h)u_kdx\\
		&\le\int_{\Omega}\left( \varepsilon[f]^\frac{N}{s}_h(x,(k-1)h)+C(\varepsilon)u_k^\frac{N}{N-s}\right) dx \\
		&\le \varepsilon\int_{\Omega}\left( \frac{1}{h}\int_{(k-1)h}^{kh}f(x,t)dt\right)^\frac{N}{s}dx+C(\varepsilon)\int_{\Omega}u_k^\frac{N}{N-s} dx\\
		&\le \frac{\varepsilon}{h^\frac{N}{s}}\left( \int_{(k-1)h}^{kh}\left( \int_{\Omega}|f(x,t)|^\frac{N}{s}dx\right)^\frac{s}{N} dt\right)^ \frac{N}{s}+C(\varepsilon,N,s,\Omega)\left( [u_k]_{W^{s,1}_0(\Omega)}\right)^\frac{N}{N-S} \\
		&\le \frac{\varepsilon}{h^\frac{N}{s}}h^\frac{N-s}{s}\int_{(k-1)h}^{kh}\int_{\Omega}|f(x,t)|^\frac{N}{s}dxdt+C(\varepsilon,N,s,\Omega,u_0)\\
		&\le \frac{\varepsilon}{h}\int_{(k-1)h}^{kh}\int_{\Omega}|f(x,t)|^\frac{N}{s}dxdt+C(\varepsilon,N,s,\Omega,u_0),
	\end{align*}
	where a Sobolev-type inequality involving the fractional norm $\|\cdot\|_{W^{s,1}_0(\Omega)}$, the Young's inequality, H\"{o}lder's inequality and \eqref{4.3} were utilized. Hence, we have
	\begin{align}
		\label{4.4}
		\frac{1}{2}\int_{\Omega}u_k^2dx\le C(f,\varepsilon,N,s,\Omega,u_0)h+\frac{1}{2}\int_{\Omega}u_{k-1}^2dx.
	\end{align}
	By adding all inequalities \eqref{4.4} for $k=1, 2, \dots, i$, we obtain
	\begin{align*}
		\frac{1}{2}\int_{\Omega}u_i^2dx\le C(f,\varepsilon,N,s,\Omega,u_0)T+\frac{1}{2}\int_{\Omega}u_{0}^2dx.
	\end{align*}
	
	Now we take a function $Z_0\in \text{sgn}(u_0(x)-u_0(y))$ and define
	\begin{align*}
		u^m(x,t):=
		\begin{cases}
			u_0(x)&\text{if }t=0,\\
			\\
			\dfrac{t}{h}u_1(x)+\dfrac{h-t}{h}u_0(x)&\text{if }0<t\le h,\\
			\\
			\dfrac{t-h}{h}u_2(x)+\dfrac{2h-t}{h}u_1(x)&\text{if }h<t\le 2h,\\
			\\
			\cdots\cdots&\cdots\cdots, \\
			\\
			\dfrac{t-(m-1)h}{h}u_m(x)+\dfrac{mh-t}{h}u_{m-1}(x)&\text{if }(m-1)h<t\le mh=T,
		\end{cases}
	\end{align*}
and
	\begin{align*}
	Z^m(x,y,t)=
	\begin{cases}
		Z_0(x,y)&\text{if }t=0,\\
		\\
		\dfrac{t}{h}Z_1(x,y)+\dfrac{h-t}{h}Z_0(x,y)&\text{if }0<t\le h,\\
		\\
		\dfrac{t-h}{h}Z_2(x,y)+\dfrac{2h-t}{h}Z_1(x,y)&\text{if }h<t\le 2h,\\
		\\
		\cdots\cdots&\cdots\cdots, \\
		\\
		\dfrac{t-(m-1)h}{h}Z_m(x,y)+\dfrac{mh-t}{h}Z_{m-1}(x,y)&\text{if }(m-1)h<t\le mh=T.
	\end{cases}
	\end{align*}
	Thus we have
	\begin{align*}
		\int_{\mathbb{R}^N\times\mathbb{R}^N}\frac{|u^m(x,t)-u^m(y,t)|}{|x-y|^{N+s}}dxdy\le C+\int_{\mathbb{R}^N\times\mathbb{R}^N}\frac{|u_0(x)-u_0(y)|}{|x-y|^{N+s}}dxdy
	\end{align*}
	and
	\begin{align*}
		\int_{\Omega}\left( u^m(x,t)\right)^2dx\le C+ \int_{\Omega}\left( u_0(x)\right)^2dx.
	\end{align*}
	After taking the supermum over $[0,T]$, we deduce
	\begin{align}
		\label{4.9}
		\sup\limits_{0\le t\le T}\left\lbrace  \int_{\Omega}\left( u^m(x,t)\right)^2dx+\int_{\mathbb{R}^N\times\mathbb{R}^N}\frac{|u^m(x,t)-u^m(y,t)|}{|x-y|^{N+s}}dxdy\right\rbrace \le C,
	\end{align}
	where $C$ is independent of $m$.
	
	Therefore, we could choose a subsequence, still denoted by $\left\lbrace u^m\right\rbrace $, such that
	\begin{align*}
		\begin{cases}
		u^m\rightharpoonup u\quad\text{weakly-}^*\text{ in }L^\infty(0,T;L^2(\Omega)),\\
		u\in L^\infty(0,T;L^2(\Omega)\cap W^{s,1}_0(\Omega)).
		\end{cases}
	\end{align*}
	Moreover, we may find a function $Z\in L^\infty(\mathbb{R}^N\times\mathbb{R}^N\times[0,T])$, such that
	\begin{align*}
		Z^m\rightharpoonup Z\quad\text{weakly-}^*\text{ in }L^\infty(\mathbb{R}^N\times\mathbb{R}^N\times[0,T]).
	\end{align*}
	
	Next, we will prove that $u$ belongs to the space $C^{0,\frac{1}{2}}\left( [0,T];L^2(\Omega)\right) $. Define $J^{s,1}_k$ on $L^2(\Omega)\cap W^{s,1}_0(\Omega)$ as follows
	\begin{align*}
		J^{s,1}_k(u):=\frac{2}{h}\int_{\mathbb{R}^N\times\mathbb{R}^N}\frac{|u(x)-u(y)|}{|x-y|^{N+s}}dxdy+\int_{\Omega}\left( \frac{u-u_{k-1}}{h}-[f]_h(x,(k-1)h)\right)^2dx.
	\end{align*}
	From the equivalence between weak solutions and minimizers, noting that the function $u_{k-1}$ is an admissible comparison map for the minimizer $u_k$, we have
	\begin{align*}
		&\quad \frac{2}{h}\int_{\mathbb{R}^N\times\mathbb{R}^N}\frac{|u_k(x)-u_k(y)|}{|x-y|^{N+s}}dxdy+\int_{\Omega}\left( \frac{u_k-u_{k-1}}{h}-[f]_h(x,(k-1)h)\right)^2dx\\
		&\le J^{s,1}_k(u_k)\le J^{s,1}_k(u_{k-1})\\
		&=\frac{2}{h}\int_{\mathbb{R}^N\times\mathbb{R}^N}\frac{|u_{k-1}(x)-u_{k-1}(y)|}{|x-y|^{N+s}}dxdy+\int_{\Omega}\left( [f]_h(x,(k-1)h)\right)^2dx.
	\end{align*}
	Thus, the following inequality holds
	\begin{align*}
		&\quad\frac{2}{h}\int_{\mathbb{R}^N\times\mathbb{R}^N}\frac{|u_k(x)-u_k(y)|}{|x-y|^{N+s}}dxdy+\frac{1}{2}\int_{\Omega}\left( \frac{u_k-u_{k-1}}{h}\right)^2dx\\
		&\le \frac{2}{h}\int_{\mathbb{R}^N\times\mathbb{R}^N}\frac{|u_k(x)-u_k(y)|}{|x-y|^{N+s}}dxdy+\int_{\Omega}\left( \frac{u_k-u_{k-1}}{h}-[f]_h(x,(k-1)h)\right)^2dx\\
		&\quad+\int_{\Omega}\left( [f]_h(x,(k-1)h)\right)^2dx\\
		&\le \frac{2}{h}\int_{\mathbb{R}^N\times\mathbb{R}^N}\frac{|u_{k-1}(x)-u_{k-1}(y)|}{|x-y|^{N+s}}dxdy+2\int_{\Omega}\left( [f]_h(x,(k-1)h)\right)^2dx.
	\end{align*}
	Adding the previous inequalities for $k=1, 2, \dots, m$, one can show that
	\begin{align*}
		&\quad\frac{2}{h}\int_{\mathbb{R}^N\times\mathbb{R}^N}\frac{|u_m(x)-u_m(y)|}{|x-y|^{N+s}}dxdy+\frac{1}{2}\sum_{k=1}^{m}\int_{\Omega}\left( \frac{u_k-u_{k-1}}{h}\right)^2dx\\
		&\le \frac{2}{h}\int_{\mathbb{R}^N\times\mathbb{R}^N}\frac{|u_0(x)-u_0(y)|}{|x-y|^{N+s}}dxdy+2\sum_{k=1}^{m}\int_{\Omega}\left( [f]_h(x,(k-1)h)\right)^2dx.
	\end{align*}
	By the uniformly boundedness of energy, we have
	\begin{align*}
		&\quad\int_{0}^{T}\int_{\Omega}\left| \partial_tu^m\right|^2dxdt\\
		&=\sum_{k=1}^{m}\int_{(k-1)h}^{kh}\int_{\Omega}\frac{1}{h^2}|u_{k}-u_{k-1}|^2dxdt\\
		& =\frac{1}{h}\sum_{k=1}^{m}\int_{\Omega}|u_{k}-u_{k-1}|^2dxdt\\
		&\le 4\left( [u^m(0)]_{W^{s,1}_0(\Omega)}+[u^m(T)]_{W^{s,1}_0(\Omega)}\right) +4h\sum_{k=1}^{m}\int_{\Omega}\left( \frac{1}{h}\int_{(k-1)h}^{kh}f(x,t)dt\right)^2dx \\
		&\le C(N,s,f,u_0,\Omega)+\frac{4}{h}\sum_{k=1}^{m}\int_{\Omega}\left( \int_{(k-1)h}^{kh}f(x,t)dt\right)^2dx\\
		&\le C(N,s,f,u_0,\Omega)+\frac{4}{h}\sum_{k=1}^{m}\left( \int_{(k-1)h}^{kh}\left( \int_{\Omega}f^2(x,t)dx\right)^\frac{1}{2}dt \right)^2\\
		&\le C(N,s,f,u_0,\Omega)+\frac{4}{h}\sum_{k=1}^{m}\left( \left( \int_{(k-1)h}^{kh}\int_{\Omega}f^2(x,t)dxdt\right)^\frac{1}{2} h^\frac{1}{2}\right)  ^2\\
		&=C(N,s,f,u_0,\Omega)+4\|f\|^2_{L^2(\Omega_T)}.
	\end{align*}
	Hence, $\left\lbrace u^m\right\rbrace $ is bounded in $W^{1,2}(0,T;L^2(\Omega))$. By Rellich's theorem, we conclude that there exists a subsequence, without loss of generality, still denoted by $u$ the limit of this subsequence,
	\begin{align*}
		\begin{cases}
			u^m\rightarrow u\quad\text{strongly in } L^2(\Omega_T),\\
			\partial_tu^m\rightharpoonup \partial_tu\quad\text{weakly}\text{ in }L^2(\Omega_T).
		\end{cases}
	\end{align*}
	In particular, noting that for any $t_1,t_2\in[0,T]$, there holds
	\begin{align*}
		\displaystyle
		\frac{\|u(t_1)-u(t_2)\|_{L^2(\Omega)}}{(t_1-t_2)^\frac{1}{2}}&=\frac{\left( \int_{\Omega}\left( u(t_1)-u(t_2)\right)^2 dx\right)^\frac{1}{2}}{(t_1-t_2)^\frac{1}{2}}\\
		&\le \frac{\left( \int_{\Omega}\left( \int_{t_2}^{t_1}\partial_tudt\right)^2 dx\right)^\frac{1}{2}}{(t_1-t_2)^\frac{1}{2}} \\
		&\le \frac{\int_{t_1}^{t_2}\left( \int_{\Omega}|\partial_tu|^2dx\right)^\frac{1}{2}dt}{(t_1-t_2)^\frac{1}{2}} \\
		&\le \left( \int_{t_1}^{t_2}\int_{\Omega}|\partial_tu|^2dxdt\right)^\frac{1}{2} ,
	\end{align*}
	which means that $u\in C^{0,\frac{1}{2}}\left( [0,T];L^2(\Omega)\right) $.
	
	Next, we will show that $u$ and $Z$ satisfy (1) in Definition \ref{def1.1}. 
	
	For each $\varphi\in C^1\left( [0,T];L^2(\Omega)\right) \cap C([0,T];W^{s,1}_0(\Omega))$ with $\varphi(\cdot,T)=0$, we take $\varphi(x,kh)$ as a test function in \eqref{4.1} for every $k\in\left\lbrace 1, 2, \dots, m\right\rbrace $ to get
	\begin{align}
		\label{4.5}
		\begin{aligned}
			&\quad\int_{\Omega}\frac{u_k(x)-u_{k-1}(x)}{h}\varphi(x,kh)dx+\int_{\mathbb{R}^N\times\mathbb{R}^N}Z_k\frac{\varphi(x,kh)-\varphi(y,kh)}{|x-y|^{N+s}}dxdy\\
			&=\int_{\Omega}[f]_h(x,(k-1)h)\varphi(x,kh)dx.
		\end{aligned}
	\end{align}
	For any $t\in (h,T)$, there exists $k\in\left\lbrace 2,3,\dots,m\right\rbrace $, such that $(k-1)h<t\le kh$. Thus, we have
	\begin{align}
		\label{4.6}
		\begin{aligned}
			&\quad\int_{\Omega} \frac{kh-t}{h}\frac{u_{k-1}(x)-u_{k-2}(x)}{h}\varphi(x,(k-1)h)dx\\
			&\quad+\int_{\Omega}\frac{t-(k-1)h}{h}\frac{u_k(x)-u_{k-1}(x)}{h}\varphi(x,kh) dx\\
			&\quad+\int_{\mathbb{R}^N\times\mathbb{R}^N} \frac{kh-t}{h}Z_{k-1}\frac{\varphi(x,(k-1)h)-\varphi(y,(k-1)h)}{|x-y|^{N+s}}dxdy\\
			&\quad+\int_{\mathbb{R}^N\times\mathbb{R}^N}\frac{t-(k-1)h}{h}Z_k\frac{\varphi(x,kh)-\varphi(y,kh)}{|x-y|^{N+s}}dxdy \\
			&=\int_{\Omega} \frac{kh-t}{h}[f]_h(x,(k-2)h)\varphi(x,(k-1)h)dx\\
			&\quad+\int_{\Omega}\frac{t-(k-1)h}{h}[f]_h(x,(k-1)h)\varphi(x,kh) dx.
		\end{aligned}
	\end{align}
That is, by definitions of $u^m$ and $Z^m$,
	\begin{align}
		\label{4.7}
		\begin{aligned}
			&\quad\int_{\Omega} \frac{u^m(x,t)-u^m(x,(t-h))}{h}\varphi(x,(k-1)h)dx\\
			&\quad+\int_{\Omega}\frac{kh-t}{h}\frac{u_k-u_{k-1}}{h}\left( \varphi(x,kh)-\varphi(x,(k-1)h)\right)  dx\\
			&\quad+\int_{\mathbb{R}^N\times\mathbb{R}^N} Z^m\frac{\varphi(x,(k-1)h)-\varphi(y,(k-1)h)}{|x-y|^{N+s}}dxdy\\
			&\quad+\int_{\mathbb{R}^N\times\mathbb{R}^N}\frac{kh-t}{h}Z_k \frac{\varphi(x,kh)-\varphi(y,kh)}{|x-y|^{N+s}} dxdy\\
			&\quad-\int_{\mathbb{R}^N\times\mathbb{R}^N}\frac{kh-t}{h}Z_k\frac{\varphi(x,(k-1)h)-\varphi(y,(k-1)h)}{|x-y|^{N+s}} dxdy\\
			&=\int_{\Omega} \frac{t-(k-1)h}{h}[f]_h(x,(k-2)h)\varphi(x,(k-1)h)dx\\
			&\quad+\int_{\Omega}\frac{kh-t}{h}[f]_h(x,(k-1)h)\varphi(x,kh) dx,
		\end{aligned}
	\end{align}
	where $k$ depends on $t$. By integrating \eqref{4.7} with respect to $t$ from $h$ to $T$, for the first term on the left-hand side of \eqref{4.7}, we get
	\begin{align}
		\label{4.11}
		\begin{aligned}
			&\quad\int_{h}^{T}\int_{\Omega}\frac{u^m(x,t)-u^m(x,(t-h))}{h}\varphi(x,(k-1)h)dxdt\\
			&=\int_{h}^{T}\int_{\Omega}\frac{u^m(x,t)}{h}\varphi(x,(k-1)h)dxdt-\int_{0}^{T-h}\int_{\Omega}\frac{u^m(x,t)}{h}\varphi(x,kh)dxdt\\
			&=\int_{h}^{T-h}\int_{\Omega}u^m(x,t)\frac{\varphi(x,(k-1)h)-\varphi(x,kh)}{h}dxdt\\
			&\quad+\int_{T-h}^{T}\int_{\Omega}\frac{u^m(x,t)}{h}\varphi(x,(m-1)h)dxdt-\int_{0}^{h}\int_{\Omega}\frac{u^m(x,t)}{h}\varphi(x,0)dxdt.
		\end{aligned}
	\end{align}
	Noting that
	\begin{align*}
		&\quad\int_{T-h}^{T}\int_{\Omega}\frac{u^m(x,t)}{h}\varphi(x,(m-1)h)dxdt\\
		&=\int_{T-h}^{T}\int_{\Omega}\left( \dfrac{t-(m-1)h}{h}u_m(x)+\dfrac{mh-t}{h}u_{m-1}(x)\right) \dfrac{\varphi(x,(m-1)h)}{h}dxdt\\
		&=\frac{1}{2}\int_{\Omega}u_m(x)\varphi(x,(m-1)h)dx+\frac{1}{2}\int_{\Omega}u_{m-1}(x)\varphi(x,(m-1)h)dx\\
		&=\frac{1}{2}\int_{\Omega}u^m(x,T)\varphi(x,(m-1)h)dx+\frac{1}{2}\int_{\Omega}u^m(x,(m-1)h)\varphi(x,(m-1)h)dx
	\end{align*}
	and
	\begin{align*}
		&\quad\int_{0}^{h}\int_{\Omega}\frac{u^m(x,t)}{h}\varphi(x,0)dxdt\\
		&=\int_{0}^{h}\int_{\Omega}\left( \dfrac{t}{h}u_1(x)+\dfrac{h-t}{h}u_0(x)\right) \frac{\varphi(x,0)}{h}dxdt\\
		&=\frac{1}{2}\int_{\Omega}u_0(x)\varphi(x,0)dx+\frac{1}{2}\int_{\Omega}u_{1}(x)\varphi(x,0)dx,
	\end{align*}
	we have
	\begin{align}
		\label{4.12}
		\begin{aligned}
			&\quad\lim\limits_{m\rightarrow+\infty}\bigg|\int_{T-h}^{T}\int_{\Omega}\frac{u^m(x,t)}{h}\varphi(x,(m-1)h)dxdt\bigg|\\
			&=\lim\limits_{m\rightarrow+\infty}\bigg| \frac{1}{2}\int_{\Omega}\left( u^m(x,T)+u^m(x,(m-1)h)\right) \varphi(x,(m-1)h)dx\bigg| \\
			&\le\lim\limits_{m\rightarrow+\infty}\frac{1}{2}\left( \|u^m(x,T)\|_{L^2(\Omega)}+\|u^m(x,T-h)\|_{L^2(\Omega)}\right) \|\varphi(x,(m-1)h)\|_{L^2(\Omega)}\\
			&\le\lim\limits_{m\rightarrow+\infty}C(u_0,f,N,s,\Omega)\|\varphi(x,(m-1)h)-\varphi(x,T)\|_{L^2(\Omega)}\\
			&\le \lim\limits_{m\rightarrow+\infty}C(u_0,f,N,s,\Omega)\max\limits_{T-h\le s\le T}\|\varphi_t(x,s)\|_{L^2(\Omega)}\frac{T}{m}\\
			&=0
		\end{aligned}
	\end{align}
	and
	\begin{align}
		\label{4.13}
		\begin{aligned}
			&\quad\lim\limits_{m\rightarrow+\infty}\bigg|\int_{0}^{h}\int_{\Omega}\frac{u^m(x,t)}{h}\varphi(x,0)dxdt-\int_{\Omega}u_0(x)\varphi(x,0)dx\bigg|\\
			&=\lim\limits_{m\rightarrow+\infty}\bigg|\frac{1}{2}\int_{\Omega}u_{1}(x)\varphi(x,0)dxdt-\frac{1}{2}\int_{\Omega}u_0(x)\varphi(x,0)dx\bigg|\\
			&\le \lim\limits_{m\rightarrow+\infty}\frac{1}{2}\|u_1(x)-u_0(x)\|_{L^2(\Omega)}\|\varphi(x,0)\|_{L^2(\Omega)}\\
			&\le \lim\limits_{m\rightarrow+\infty}C(u_0,f,s,N,\Omega)h^\frac{1}{2}\|\varphi(x,0)\|_{L^2(\Omega)}\\
			&=0.
		\end{aligned}
	\end{align}
	Observing that
	\begin{align}
		\label{4.14}
		\begin{aligned}
			&\quad\lim\limits_{m\rightarrow+\infty}\bigg|\int_{h}^{T-h}\int_{\Omega}u^m(x,t)\frac{\varphi(x,(k-1)h)-\varphi(x,kh)}{h}dxdt-\int_{0}^{T}\int_{\Omega}u(x,t)\varphi_t(x,t)dxdt\bigg|\\
			&\le \lim\limits_{m\rightarrow+\infty}\bigg| \int_{h}^{T-h}\int_{\Omega}\left( u^m(x,t)\frac{\varphi(x,(k-1)h)-\varphi(x,kh)}{h}-u^m(x,t)\varphi_t(x,t)\right)dxdt \bigg|\\
			&\quad+\lim\limits_{m\rightarrow+\infty}\bigg|\int_{h}^{T-h}\int_{\Omega}u^m(x,t)\varphi_t(x,t)dxdt-\int_{0}^{T}\int_{\Omega}u^m(x,t)\varphi_t(x,t)dxdt\bigg|\\
			&\quad+\lim\limits_{m\rightarrow+\infty}\bigg|\int_{0}^{T}\int_{\Omega}\left( u^m(x,t)-u(x,t)\right) \varphi_t(x,t)dxdt\bigg|\\
			&\le \lim\limits_{m\rightarrow+\infty}\int_{h}^{T-h}\int_{\Omega}|u^m(x,t)|\bigg|\frac{\varphi(x,(k-1)h)-\varphi(x,kh)}{h}-\varphi_t(x,t)\bigg|dxdt\\
			&\quad+\lim\limits_{m\rightarrow+\infty}\left( \int_{0}^{h}\int_{\Omega}|u^m(x,t)\varphi_t(x,t)|dxdt+\int_{T-h}^{T}\int_{\Omega}|u^m(x,t)\varphi_t(x,t)|dxdt\right) \\
			&\le \lim\limits_{m\rightarrow+\infty}\int_{0}^{T}\|u^m(x,t)\|_{L^2(\Omega)}\max\limits_{(k-1)h\le \tau\le kh}\|\varphi_t(x,\tau)-\varphi_t(x,t)\|_{L^2(\Omega)}dt\\
			&\quad+\lim\limits_{m\rightarrow+\infty}2\|u^m\|_{L^\infty(0,T;L^2(\Omega))}\|\varphi_t\|_{L^\infty(0,T;L^2(\Omega))}h\\
			&=0,
		\end{aligned}
	\end{align}
	by letting $m$ goes to $+\infty$ and bringing \eqref{4.12}--\eqref{4.14} into \eqref{4.11} one has
	\begin{align}
		\label{4.15}
		\begin{aligned}
			&\quad\lim\limits_{m\rightarrow+\infty}\int_{h}^{T}\int_{\Omega}\frac{u^m(x,t)-u^m(x,(t-h))}{h}\varphi(x,(k-1)h)dxdt\\
			&=-\int_{0}^{T}\int_{\Omega}u(x,t)\varphi_t(x,t)dxdt-\int_{\Omega}u_0(x)\varphi(x,0)dx.
		\end{aligned}
	\end{align}
	 
	For the second term on the left-hand side of \eqref{4.7}, from the boundedness of $\left\lbrace u^m\right\rbrace $ in the space $C^{0,\frac{1}{2}}\left( [0,T];L^2(\Omega)\right) $, we have
	\begin{align}
		\label{4.16}
		\begin{aligned}
			&\quad\lim\limits_{m\rightarrow+\infty}\bigg|\int_{h}^{T}\int_{\Omega}\frac{kh-t}{h}\frac{u_k-u_{k-1}}{h}\left( \varphi(x,kh)-\varphi(x,(k-1)h)\right)  dxdt\bigg|\\
			&\le \lim\limits_{m\rightarrow+\infty}\bigg|\int_{h}^{T}\int_{\Omega}\frac{u_k-u_{k-1}}{h}\left( \varphi(x,kh)-\varphi(x,(k-1)h)\right)  dxdt\bigg|\\
			&\le \lim\limits_{m\rightarrow+\infty}\int_{h}^{T}\bigg\|\frac{u_k-u_{k-1}}{h}\bigg\|_{L^2(\Omega)}\|\varphi(x,kh)-\varphi(x,(k-1)h)\|_{L^2(\Omega)}dt\\
			&\le \lim\limits_{m\rightarrow+\infty}\int_{h}^{T}C(u_0,f,s,N,\Omega)h^\frac{1}{2}\max\limits_{h\le \tau\le T}\|\varphi_t(x,\tau)\|_{L^2(\Omega)}dt\\
			&\le \lim\limits_{m\rightarrow+\infty}C(u_0,f,s,N,\Omega,\varphi,T)h^\frac{1}{2}\\
			&=0.
		\end{aligned}
	\end{align}
	
	Similarly, for the third term on the left-hand side of \eqref{4.7}, we deduce that
	\begin{align*}
		&\quad\lim\limits_{m\rightarrow+\infty}\bigg|\int_{h}^{T}\int_{\mathbb{R}^N\times\mathbb{R}^N} Z^m\frac{\varphi(x,(k-1)h)-\varphi(y,(k-1)h)}{|x-y|^{N+s}}dxdydt\\
		&\qquad\qquad\quad-\int_{0}^{T}\int_{\mathbb{R}^N\times\mathbb{R}^N}Z\frac{\varphi(x,t)-\varphi(y,t)}{|x-y|^{N+s}} dxdydt\bigg|\\
		&\le \lim\limits_{m\rightarrow+\infty}\bigg|\int_{h}^{T}\int_{\mathbb{R}^N\times\mathbb{R}^N} Z^m\frac{\varphi(x,(k-1)h)-\varphi(y,(k-1)h)}{|x-y|^{N+s}}dxdydt\\
		&\qquad\qquad\quad-\int_{0}^{T}\int_{\mathbb{R}^N\times\mathbb{R}^N}Z^m\frac{\varphi(x,t)-\varphi(y,t)}{|x-y|^{N+s}} dxdydt\bigg|\\
		&\quad+\lim\limits_{m\rightarrow+\infty}\bigg|\int_{0}^{T}\int_{\mathbb{R}^N\times\mathbb{R}^N}\left( Z^m-Z\right) \frac{\varphi(x,t)-\varphi(y,t)}{|x-y|^{N+s}} dxdydt\bigg|\\
		&\le \lim\limits_{m\rightarrow+\infty}\int_{0}^{T}\int_{\mathbb{R}^N\times\mathbb{R}^N}\bigg|\frac{\varphi(x,(k-1)h)-\varphi(y,(k-1)h)}{|x-y|^{N+s}}-\frac{\varphi(x,t)-\varphi(y,t)}{|x-y|^{N+s}}\bigg|dxdydt\\
		&=0
	\end{align*}
	and for the right-hand side of \eqref{4.7}, we find
	\begin{align*}
		&\quad \lim\limits_{m\rightarrow+\infty}\int_{h}^{T}\int_{\Omega} \frac{t-(k-1)h}{h}[f]_h(x,(k-2)h)\varphi(x,(k-1)h)dxdt\\
		&\quad+\lim\limits_{m\rightarrow+\infty}\int_{h}^{T}\int_{\Omega}\frac{kh-t}{h}[f]_h(x,(k-1)h)\varphi(x,kh) dxdt\\
		&=\lim\limits_{m\rightarrow+\infty}\int_{h}^{T-h}\int_{\Omega}[f]_h(x,(k-1)h)\varphi(x,kh)dxdt+\lim\limits_{m\rightarrow+\infty}\int_{0}^{h}\int_{\Omega}\frac{t}{h}[f]_h(x,0)\varphi(x,h)dxdt\\
		&\quad +\lim\limits_{m\rightarrow+\infty}\int_{T-h}^{T}\int_{\Omega}\frac{T-t}{h}[f]_h(x,T-h)\varphi(x,T)dxdt.
 	\end{align*}
	Due to
	\begin{align*}
		&\quad\lim\limits_{m\rightarrow+\infty}\bigg|\int_{h}^{T-h}\int_{\Omega}[f]_h(x,(k-1)h)\varphi(x,kh)dxdt-\int_{0}^{T}\int_{\Omega}f\varphi dxdt\bigg|\\
		&\le \lim\limits_{m\rightarrow+\infty}\bigg|\int_{h}^{T-h}\int_{\Omega}[f]_h(x,(k-1)h)\varphi(x,kh)dxdt-\int_{0}^{T}\int_{\Omega}f\varphi(x,kh)dxdt\bigg|\\
		&\quad+\lim\limits_{m\rightarrow+\infty}\bigg|\int_{0}^{T}\int_{\Omega}f\varphi(x,kh)\varphi dxdt-\int_{0}^{T}\int_{\Omega}f\varphi dxdt\bigg|\\
		&=0
	\end{align*}
	and
	\begin{align*}
		&\quad\lim\limits_{m\rightarrow+\infty}\int_{0}^{h}\int_{\Omega}\frac{t}{h}[f]_h(x,0)\varphi(x,h)dxdt+\lim\limits_{m\rightarrow+\infty}\int_{T-h}^{T}\int_{\Omega}\frac{T-t}{h}[f]_h(x,T-h)\varphi(x,T)dxdt\\
		&=\lim\limits_{m\rightarrow+\infty}\frac{h}{2}\int_{\Omega}\frac{1}{h}\int_{0}^{h}f(x,\tau)d\tau\varphi(x,h)dx+\lim\limits_{m\rightarrow+\infty}\frac{h}{2}\int_{\Omega}\frac{1}{h}\int_{T-h}^{T}f(x,\tau)d\tau\varphi(x,T)dx\\
		&=\frac{1}{2}\lim\limits_{m\rightarrow+\infty}\int_{0}^{h}\int_{\Omega}f(x,t)\varphi(x,h)dxdt\\
		&=0,
	\end{align*}
	one has
	\begin{align}
		\label{4.17}
		\begin{aligned}
			&\quad\lim\limits_{m\rightarrow+\infty}\int_{h}^{T}\int_{\mathbb{R}^N\times\mathbb{R}^N} Z^m\frac{\varphi(x,(k-1)h)-\varphi(y,(k-1)h)}{|x-y|^{N+s}}dxdydt\\
			&=\int_{0}^{T}\int_{\mathbb{R}^N\times\mathbb{R}^N}Z\frac{\varphi(x,t)-\varphi(y,t)}{|x-y|^{N+s}} dxdydt
		\end{aligned}
	\end{align}
	and
	\begin{align}
		\label{4.18}
		\begin{aligned}
			&\quad \lim\limits_{m\rightarrow+\infty}\int_{h}^{T}\int_{\Omega} \frac{t-(k-1)h}{h}[f]_h(x,(k-2)h)\varphi(x,(k-1)h)dxdt\\
			&\quad+\lim\limits_{m\rightarrow+\infty}\int_{h}^{T}\int_{\Omega}\frac{kh-t}{h}[f]_h(x,(k-1)h)\varphi(x,kh) dxdt\\
			&=\int_{0}^{T}\int_{\Omega}f\varphi dxdt.
		\end{aligned}
	\end{align}
	Combining \eqref{4.15}--\eqref{4.18}, we get
	\begin{align*}
		\begin{aligned}
			&\quad-\int_{\Omega}u_0\varphi(x,0)dx-\int_{0}^{T}\int_{\Omega}u\varphi_tdxdt+\int_{0}^{T}\int_{\mathbb{R}^N\times\mathbb{R}^N}Z\frac{\varphi(x,t)-\varphi(y,t)}{|x-y|^{N+s}}dxdydt\\
			&=\int_{0}^{T}\int_{\Omega}f\varphi dxdt.
		\end{aligned}
	\end{align*}
	
	To conclude the proof, it remains to show that $Z(x,y,t)\in \text{sgn}\left( u(x,t)-u(y,t)\right) $, i.e., for almost all $t\in[0,T]$ and $(x,y)\in\mathbb{R}^N\times\mathbb{R}^N$, $Z(x,y,t)=\frac{u(x,t)-u(y,t)}{|u(x,t)-u(y,t)|}$ in the domain $\left\lbrace u(x,t)\neq u(y,t)\right\rbrace $ and $\|Z\|_{L^\infty(\mathbb{R}^N\times\mathbb{R}^N\times[0,T])}\le1$.
	
	Choosing $u$ as a test function in \eqref{1.4} to deduce that
	\begin{align}
		\label{4.19}
		\begin{aligned}
			&\quad\frac{1}{2}\|u(T)\|^2_{L^2(\Omega)}-\frac{1}{2}\|u_0\|^2_{L^2(\Omega)}+\int_{0}^{T}\int_{\mathbb{R}^N\times\mathbb{R}^N}Z\frac{u(x,t)-u(y,t)}{|x-y|^{N+s}}dxdydt\\
			&=\int_{0}^{T}\int_{\Omega}fudxdt.
		\end{aligned}
	\end{align}
	Moreover, by choosing $\varphi=u_k$ in \eqref{4.1}, we obtain
	\begin{align*}
		\int_{\Omega}\frac{u_k-u_{k-1}}{h}u_kdx+\int_{\mathbb{R}^N\times\mathbb{R}^N}\frac{|u_k(x)-u_k(y)|}{|x-y|^{N+s}}dxdy=\int_{\Omega}[f]_h(x,(k-1)h)u_kdx.
	\end{align*}
	Thus, for any $t\in [(k-1)h,kh]$, we have
	\begin{align}
		\label{4.20}
		\begin{aligned}
			&\quad\frac{t-(k-1)h}{h}\left[ \int_{\Omega}\frac{u_k-u_{k-1}}{h}u_kdx+\int_{\mathbb{R}^N\times\mathbb{R}^N}\frac{|u_k(x)-u_k(y)|}{|x-y|^{N+s}}dxdy\right]\\
			&\quad+\frac{kh-t}{h}\left[ \int_{\Omega}\frac{u_{k-1}-u_{k-2}}{h}u_{k-1}dx+\int_{\mathbb{R}^N\times\mathbb{R}^N}\frac{|u_{k-1}(x)-u_{k-1}(y)|}{|x-y|^{N+s}}dxdy\right]\\
			&=\int_{\Omega}\left[ \frac{t-(k-1)h}{h}[f]_h(x,(k-2)h)u_{k-1}+\frac{kh-t}{h}[f]_h(x,(k-1)h)u_k\right] dx.
		\end{aligned}
	\end{align}
	Integrating \eqref{4.20} with respect to $t$ from $(k-1)h$ to $kh$ to obtain
	\begin{align}
		\label{4.21}
		\begin{aligned}
			&\quad\frac{1}{2}\int_{\Omega}(u_k-u_{k-1})u_kdx+\frac{t-(k-1)h}{h}\int_{(k-1)h}^{kh}\int_{\mathbb{R}^N\times\mathbb{R}^N}\frac{|u_k(x)-u_k(y)|}{|x-y|^{N+s}}dxdydt\\
			&\quad+\frac{1}{2}\int_{\Omega}(u_{k-1}-u_{k-2})u_{k-1}dx\\
			&\quad+\frac{kh-t}{h}\int_{(k-1)h}^{kh}\int_{\mathbb{R}^N\times\mathbb{R}^N}\frac{|u_{k-1}(x)-u_{k-1}(y)|}{|x-y|^{N+s}}dxdydt\\
			&=\int_{(k-1)}^{kh}\int_{\Omega}\left[ \frac{t-(k-1)h}{h}[f]_h(x,(k-2)h)u_{k-1}+\frac{kh-t}{h}[f]_h(x,(k-1)h)u_k\right] dxdt.
		\end{aligned}
	\end{align}
	By adding all the equation \eqref{4.21} for $k\in\left\lbrace 2,3,\dots,m\right\rbrace $, from Young's inequality and definition of $u^m$, we have
	\begin{align*}
		&\quad\frac{1}{4}\int_{\Omega}(u_m^2+u_{m-1}^2)dx+\int_{h}^{T}\int_{\mathbb{R}^N\times\mathbb{R}^N}\frac{|u^m(x,t)-u^m(y,t)|}{|x-y|^{N+s}}dxdydt\\
		&\le \sum_{k=2}^{m}\int_{(k-1)}^{kh}\int_{\Omega}\left[ \frac{t-(k-1)h}{h}[f]_h(x,(k-2)h)u_{k-1}+\frac{kh-t}{h}[f]_h(x,(k-1)h)u_k\right] dxdt\\
		&\quad+\frac{1}{4}\int_{\Omega}(u_1^2+u_0^2)dx.
	\end{align*}
	Letting $m$ goes to $+\infty$ and taking the fact that $u^m\rightarrow u$ in $C\left( [0,T];L^2(\Omega)\right) $ into consideration, we arrive at
	\begin{align}
		\label{4.22}
		\begin{aligned}
			&\quad\frac{1}{2}\int_{\Omega}u(T)^2dx+\int_{0}^{T}\int_{\mathbb{R}^N\times\mathbb{R}^N}\frac{|u(x,t)-u(y,t)|}{|x-y|^{N+s}}dxdydt\\
			&\le \int_{0}^{T}\int_{\mathbb{R}^N\times\mathbb{R}^N}fudx+\frac{1}{2}\int_{\Omega}u_0^2dx.
		\end{aligned}
	\end{align}
	Combining \eqref{4.22} and \eqref{4.19}, there holds
	\begin{align*}
		\int_{0}^{T}\int_{\mathbb{R}^N\times\mathbb{R}^N}\frac{|u(x,t)-u(y,t)|}{|x-y|^{N+s}}dxdydt\le\int_{0}^{T}\int_{\mathbb{R}^N\times\mathbb{R}^N}Z\frac{u(x,t)-u(y,t)}{|x-y|^{N+s}}dxdydt,
	\end{align*}
	which implies $Z(x,y,t)\in \text{sgn}\left( u(x,t)-u(y,t)\right) $.
\end{proof}
\begin{remark}
	By an approximation argument, we may take $\varphi=u-v$ in the formulation
	\begin{align*}
		&\quad -\int_{0}^{T}\int_{\Omega}(u-v)\varphi_tdxdt+\int_{\Omega}\left( u(T)-v(T)\right) \varphi(x,T)dxdt\\
		&=-\int_{0}^{T}\int_{\mathbb{R}^N\times\mathbb{R}^N} \left( Z_u-Z_v\right) \frac{\varphi(x,t)-\varphi(y,t)}{|x-y|^{N+s}}dxdydt
	\end{align*}
	to deduce
	\begin{align*}
		0&\le \frac{1}{2}\int_{\Omega}\left( u(T)-v(T)\right)^2dx\\
		&=-\int_{0}^{T}\int_{\mathbb{R}^N\times\mathbb{R}^N} \left( Z_u-Z_v\right) \frac{u(x,t)-v(x,t)-u(y,t)+v(y,t)}{|x-y|^{N+s}}dxdydt\le 0,
	\end{align*}
	where $Z_u\in\mathrm{sgn}\left( u(x,t)-u(y,t)\right) $ and $Z_v\in\mathrm{sgn}\left( v(x,t)-v(y,t)\right)$. The uniqueness of weak solutions to problem \eqref{1.1} follows immediately.
\end{remark}

	\medskip
	
\subsection*{Data availability}
No data was used for the research described in the article.
	
\subsection*{Conflict of interest} The authors have no non-financial Conflict of interest to declare that
are relevant to the consent of this article.

\subsection*{Acknowledgments}
This work was supported by the National Natural Science Foundation of China (No. 12071098) and the Fundamental Research Funds for the Central Universities (No. 2022FRFK060022).

\end{document}